\documentclass[reqno]{amsart}
\usepackage{amsmath,amsthm,amsfonts,amssymb,amscd}
\usepackage{graphicx,color}	
\usepackage{enumerate,tikz}
\usepackage[ocgcolorlinks, linkcolor=red!80!black, citecolor=blue!80!black]{hyperref}

\newtheorem{thm}{Theorem}[section]
\newtheorem*{thm*}{Theorem}
\newtheorem{lem}[thm]{Lemma}
\theoremstyle{definition} \newtheorem{defn}[thm]{Definition}
\newtheorem{ex}[thm]{Example}
\newtheorem*{lem*}{Lemma}
\newtheorem*{rem}{Remark}
\newtheorem*{conj*}{Conjecture}
\newtheorem{prop}[thm]{Proposition}

\newtheorem{conj}[thm]{Conjecture}
\newtheorem*{cor*}{Corollary}

\newtheorem*{warn*}{Warning}

\newcommand{\CC}{\mathbb{C}}

\newcommand{\ZZ}{\mathbb{Z}}

\renewcommand{\dim}{\operatorname{dim}}
\newcommand{\codim}{\operatorname{codim}}

\newcommand{\GL}{\operatorname{GL}}

\renewcommand{\epsilon}{\varepsilon}
\newcommand{\sgn}{\operatorname{sgn}}

\newcommand{\Gr}{\operatorname{Gr}}
\newcommand{\Fl}{\operatorname{Fl}}
\newcommand{\Red}{\operatorname{Red}}
\newcommand{\affS}{\tilde{S}}

\DeclareMathOperator{\ann}{ann}

\newcommand{\av}{\operatorname{av}}
\newcommand{\Bound}{\operatorname{Bound}}
\newcommand{\Prj}{\operatorname{Prj}}
\newcommand{\rowspan}{\operatorname{rowspan}}
\newcommand{\dom}{\operatorname{dom}}

\begin{document}

\date{\today}
\author{Brendan Pawlowski}
\thanks{The author was partially supported by grant DMS-1101017 from the NSF}
\title{Cohomology classes of interval positroid varieties and a conjecture of Liu}

\begin{abstract}

To each finite subset of $\ZZ^2$ (a diagram), one can associate a subvariety of a complex Grassmannian (a diagram variety), and a representation of a symmetric group (a Specht module). Liu has conjectured that the cohomology class of a diagram variety is represented by the Frobenius characteristic of the corresponding Specht module. We give a counterexample to this conjecture.

However, we show that for the diagram variety of a permutation diagram, Liu's conjectured cohomology class $\sigma$ is at least an upper bound on the actual class $\tau$, in the sense that $\sigma - \tau$ is a nonnegative linear combination of Schubert classes. To do this, we exhibit the appropriate diagram variety as a component in a degeneration of one of Knutson's interval positroid varieties (up to Grassmann duality). A priori, the cohomology classes of these interval positroid varieties are represented by affine Stanley symmetric functions. We give a different formula for these classes as ordinary Stanley symmetric functions, one with the advantage of being Schur-positive and compatible with inclusions between Grassmannians.

\end{abstract}

\maketitle

\section{Introduction}

\subsection{Diagram varieties}

A \emph{diagram} is a finite subset $D$ of $\ZZ^2$. Write $[n]$ for $\{1, 2, \ldots, n\}$. Given a diagram contained in $[k] \times [n-k]$, define a subvariety $X_D$ of the Grassmannian $\Gr_k(n)$ of $k$-planes in $\CC^n$ as the Zariski closure of
\begin{equation*}
\left\{\rowspan\, [A \mid I_k] : A \in M_{k, n-k} \text{ with $A_{ij} = 0$ when $(i,j) \in D$} \right \}. 
\end{equation*}
Here $M_{k,n-k}$ is the set of $k \times (n-k)$ complex matrices, and $I_k$ is the $k \times k$ identity matrix. Call this variety $X_D$ a \emph{diagram variety}. For example, if $D = \{(1,1), (1,2), (2,1)\}$, $k = 2$, $n = 4$, then $X_D$ is the closure of the set of $2$-planes in $\CC^4$ which are the rowspans of matrices of the form
\begin{equation*}
\begin{bmatrix}
0 & 0 & 1 & 0\\
0 & * & 0 & 1
\end{bmatrix}.
\end{equation*}

Let $S_D$ denote the set of permutations of $D$. One can also associate a (complex) representation $Sp_D$ of the symmetric group $S_D$ to a diagram $D$, called the \emph{Specht module of $D$}. These generalize the usual irreducible Specht modules, which occur when $D$ is the Young diagram of a partition; the definition for general diagrams is due to James and Peel \cite{jamespeel}.

Each of these objects, diagram variety and Specht module, naturally leads to a class in the cohomology ring $H^*\Gr_k(n) := H^*(\Gr_k(n), \ZZ)$. For the diagram variety, we take the Chow ring class of $X_D$ and use the natural isomorphism between $H^*(\Gr_k(n), \ZZ)$ and the Chow ring of $\Gr_k(n)$ to obtain a cohomology class $[X_D] \in H^{2\#D}(\Gr_k(n), \ZZ)$.

As for the Specht module, let $s_D$ be the Frobenius characteristic of $S^D$, meaning $s_D = \sum_{\lambda} a_{\lambda} s_{\lambda}$ if $S^D \simeq \bigoplus_\lambda a_{\lambda} S^{\lambda}$, where $s_{\lambda}$ is a Schur function. Here $\lambda$ runs over partitions, and $S^{\lambda}$ is an irreducible Specht module. There is a surjective ring homomorphism $\phi$ from the ring of symmetric functions to $H^*(\Gr_k(n), \ZZ)$, sending the Schur function $s_{\lambda}$ to the Schubert class $\sigma_\lambda := [X_{\lambda}]$, or to $0$ if $\lambda \not\subseteq (k^{n-k})$ \cite{youngtableaux}. Hence we can consider the cohomology class $\phi(s_D)$.

% comment: label as Conjecture 2.4 here?
\begin{conj*}[Liu \cite{liuthesis}, Conjecture~\ref{conj:liu} below] For any diagram $D$, the cohomology classes $[X_D]$ and $\phi(s_D)$ are equal. \end{conj*}

Liu proved Conjecture~\ref{conj:liu}, or the weaker variant claiming equality of degrees, for various classes of diagrams \cite{liuthesis}. However, it turns out that this conjecture fails in general, as we show in Section~\ref{sec:counterexample}.

\begin{thm*} Conjecture~\ref{conj:liu} fails for $X_D \subseteq \Gr_4(8)$ where $D = \{(1,1),(2,2),(3,3),(4,4)\}$. \end{thm*}

Let $D(w)$ denote the \emph{Rothe diagram} of $w \in S_n$: the diagram with a cell $(i, w(j))$ for each inversion $i < j$, $w(i) > w(j)$ of $w$. Work of Kra\'skiewicz and Pragacz \cite{kraskiewicz-pragacz} and of Reiner and Shimozono \cite{plactification} shows that $s_{D(w)}$ is the Stanley symmetric function $F_w$ \cite{stanleysymm}. Thus, if Conjecture~\ref{conj:liu} were to hold for $D(w)$, we would have $[X_{D(w)}] = \phi(F_w)$.

Building on work of Postnikov \cite{postnikov-positroids}, Knutson, Lam, and Speyer \cite{positroidjuggling} have defined a collection of subvarieties $\Pi_f$ of Grassmannians called \emph{positroid varieties}, indexed by certain affine permutations $f$. A positroid variety is defined by imposing some rank conditions on cyclic intervals of columns of matrices representing points in $\Gr_k(n)$, and any irreducible variety defined by such rank conditions is a positroid variety.  They show that the positroid variety $\Pi_f$ has cohomology class $\phi(\tilde{F}_f)$, where $\tilde{F}_f$ is the affine Stanley symmetric function of $f$, as defined in \cite{lam-affine-stanley}. Given an ordinary permutation $w \in S_n$, define a bijection $f_w : \ZZ \to \ZZ$ by
\begin{equation*}
f_w(i) = \begin{cases}
i+n & \text{if $1 \leq i \leq n$}\\
w(i)+2n & \text{if $n \leq i \leq 2n$}
\end{cases}
\end{equation*}
and $f(i+2n) = f(i)+2n$. One can show that $\tilde{F}_{f_w} = F_w$.

By the previous two paragraphs, Conjecture~\ref{conj:liu} would imply equality of the classes $[\Pi_{f_w}]$ and $[X_{D(w)}]$. As we will see, Conjecture~\ref{conj:liu} can fail for permutation diagrams $D = D(w)$, and in general $[\Pi_{f_w}]$ and $[X_{D(w)}]$ need not be equal. Nevertheless, we will give a degeneration of $\Pi_{f_w}$ to a (possibly reducible) variety containing $X_{D(w)}$ as a component, which implies the following upper bound on $[X_{D(w)}]$.

\begin{thm*}[Theorem~\ref{thm:diagram-variety-class-bound}] The cohomology class $\phi(F_w) - [X_{D(w)}]$ is a nonnegative integer combination of Schubert classes. \end{thm*}

\subsection{Limits of classes of interval positroid varieties}

The positroid varieties defined by rank conditions only involving honest intervals of columns (as opposed to cyclic intervals) are called \emph{interval positroid varieties} \cite{knutson-interval-positroid-varieties}. For $w \in S_n$, the Grassmann duals of the varieties $\Pi_{f_w}$ described above are examples of interval positroid varieties. There are several ways to compute the class $[\Sigma]$ of an interval positroid variety $\Sigma$. First, $[\Sigma] = \phi(\tilde{F}_f)$ for some affine permutation $f$ by the work of Knutson-Lam-Speyer described above. Second, Coskun \cite{two-step-flags} gives a recursive rule for computing $[\Sigma]$ by degenerating $\Sigma$ to a union of Schubert varieties, and in \cite{knutson-interval-positroid-varieties}, Knutson computes the more general torus-equivariant K-theory class of $\Sigma$ in this way.

We give a different formula for $[\Sigma]$ which is \emph{stable} in the following sense. Given a list $M = (S_1, \ldots, S_m)$ of intervals all contained in $[n]$ and a vector $r = (r_1, \ldots, r_m)$ of nonnegative integers, define
\begin{equation*}
\Sigma_{M,r,n} = \{\rowspan(A) \in \Gr_k(n) : \text{the submatrix of $A$ in columns $S_i$ has rank $\leq r_i$ for all $i$}\}.
\end{equation*}
If $\Sigma_{M,r,n}$ is irreducible, then it is an interval positroid variety.

The standard inclusion $\CC^n \hookrightarrow \CC^{n+1}$ defines an inclusion $\Gr_k(\CC^n) \hookrightarrow \Gr_k(\CC^{n+1})$, hence a pullback map $H^* \Gr_k(n{+}1) \twoheadrightarrow H^* \Gr_k(n)$, and this pullback sends $[\Sigma_{M,r,n{+}1}]$ to $[\Sigma_{M,r,n}]$. We can therefore ask for a formula for a class $\alpha$ in the inverse limit $\varprojlim_N H^* \Gr_k(N)$ which represents the classes $[\Sigma_{M,r,n}]$ for every $n$, in the sense that for every $n$ the map $\displaystyle \varprojlim H^* \Gr_k(N) \to H^*\Gr_k(n)$ sends $\alpha$ to $[\Sigma_{M,r,n}]$.

\begin{thm*}[Theorem~\ref{thm:interval-positroid-class}] If $\Sigma_{M,r,n} \subseteq \Gr_k(n)$ is an interval positroid variety, there is an ordinary permutation $w$ such that the ordinary Stanley symmetric function $F_w$ represents the class $\Sigma_{M,r,n}$ for all $n$. \end{thm*}

\subsection*{Acknowledgements}

The results here owe much to conversations with Sara Billey, who suggested this problem to me. I'm also grateful to Dave Anderson, Andy Berget, Izzet Coskun, Allen Knutson, Ricky Liu, Vic Reiner, Robert Smith, Alex Yong, and an anonymous referee for helpful discussions and comments.

\section{A counterexample to Liu's conjecture}
\label{sec:counterexample}

\begin{defn}
A \emph{diagram} is a finite subset of $\ZZ^2$.
\end{defn}

Given a diagram $D$ contained in $[k] \times [n-k]$, define an open subset
\begin{equation*}
X_D^\circ = \{\rowspan\,[A\,|\,I_k] : A \in M_{k,n-k} \text{ such that $A_{ij} = 0$ whenever $(i,j) \in D$} \}
\end{equation*}
of the complex Grassmannian $\Gr_k(n)$. For example, if $D = \{(1,1), (1,2), (2,2), (2,3)\}$, $k = 2$, and $n = 5$, then
\begin{equation*}
X_D^\circ = \left \{\rowspan \begin{pmatrix}
0 & 0 & * & 1 & 0\\
* & 0 & 0 & 0 & 1
\end{pmatrix} \right \}.
\end{equation*}

\begin{defn}
The \emph{diagram variety} $X_D$ of $D$ is $\overline{X_D^\circ}$, the closure being in the Zariski topology.
\end{defn}

Notice that $X_D^{\circ}$ is an open dense subset of $X_D$ isomorphic to $\mathbb{C}^{k(n-k)-\#D}$. In particular, it is irreducible, so $X_D$ is also irreducible and has codimension $\#D$.

We now describe a representation of $S_{D}$ associated to each diagram $D$. Let $R(D)$ denote the group of permutations $\sigma \in S_D$ for which $b$ and $\sigma(b)$ are in the same row for any $b \in D$. Let $C(D)$ be the analogous subgroup with ``row'' replaced by ``column''.

\begin{defn} The \emph{Specht module} of $D$ is the left ideal
\begin{equation*}
Sp_D = \CC[S_D] \sum_{p \in R(D)} \sum_{q \in C(D)} \sgn(q)qp
\end{equation*}
of $\CC[S_D]$, viewed as an $S_D$-module.
\end{defn}

The Specht modules associated to general diagrams were studied by James and Peel \cite{jamespeel}. As $D$ runs over (Ferrers diagrams of) partitions of $m$, the Specht modules provide a complete, irredundant set of complex irreducibles for $S_m$ (see \cite{youngtableaux, sagan}). The isomorphism type of $Sp_D$ is unaltered by permuting the rows or the columns of $D$.  If the rows and columns of $D$ cannot be permuted to obtain a partition---equivalently, the rows of $D$ are not totally ordered under inclusion---then $Sp_D$ will not be irreducible. For example, if $\lambda\setminus \mu$ is a skew shape, then
\begin{equation*}
Sp_{\lambda\setminus \mu} \simeq \bigoplus_{\nu} c^{\lambda}_{\mu \nu} Sp_{\nu},
\end{equation*}
where $c^{\lambda}_{\mu \nu}$ is a Littlewood-Richardson coefficient.

In general it is an open problem to give a combinatorial rule for decomposing $Sp_{D}$ into irreducibles. The widest class of diagrams for which such a rule is known are the \emph{percent-avoiding} diagrams, studied by Reiner and Shimozono \cite{percentavoiding}; see also \cite{liuforests} and \cite{columnconvex}.

Given a diagram $D \subset [k] \times [n-k]$, let $D^{\vee}$ be the complement of $D$ in $[k] \times [n-k]$ rotated by $180^\circ$. For example, if $\mu \subseteq \lambda \subseteq [k] \times [n-k]$ are partitions, then $X_{\lambda^\vee}^\circ \cap X_{\mu}^\circ = X_{(\lambda/\mu)^\vee}^\circ$. This intersection is transverse on the dense open subset $X_{(\lambda/\mu)^\vee}^\circ$ of $X_{(\lambda/\mu)^\vee}$, and indeed one can show that $[X_{(\lambda/\mu)^\vee}] = \sum_{\nu} c^{\lambda}_{\mu \nu} \sigma_{\nu^{\vee}}$ \cite[Proposition 5.5.3]{liuthesis}.

Magyar has shown that Specht module decompositions behave as nicely as possible with respect to the box complement operation.
\begin{thm}[Magyar \cite{magyarborelweil}] \label{thm:magyar}
For any diagram $D$ contained in $[k] \times [n-k]$, $Sp_D \simeq \bigoplus_{\lambda} a_{\lambda} Sp_{\lambda}$ if and only if $Sp_{D^\vee} \simeq \bigoplus_{\lambda} a_{\lambda} Sp_{\lambda^{\vee}}$.
\end{thm}
In particular, $Sp_{(\lambda/\mu)^\vee} \simeq \bigoplus_{\nu} c_{\mu \nu}^{\lambda} Sp_{\nu^\vee}$. Comparing this decomposition of $Sp_{(\lambda/\mu)^{\vee}}$ to the expansion $[X_{(\lambda/\mu)^\vee}] = \sum_{\nu} c^{\lambda}_{\mu \nu} \sigma_{\nu^{\vee}}$ discussed above suggests the next conjecture (and proves it when $D = (\lambda/\mu)^\vee$).

\begin{conj}[Liu \cite{liuthesis}] \label{conj:liu} For any diagram $D$, the cohomology classes $[X_D]$ and $\phi(s_D)$ are equal. \end{conj}

Liu proved Conjecture~\ref{conj:liu} in the case above where $D^{\vee}$ is a skew shape, or when it corresponds to a forest \cite{liuthesis} in the sense that one can represent a diagram $D \subset [k] \times [n-k]$ as the bipartite graph with white vertices $[k]$, black vertices $[n-k]$, and an edge between a white $i$ and black $j$ whenever $(i,j) \in D$. In \cite{billey-pawlowski-2013}, we proved Conjecture~\ref{conj:liu} when $D^{\vee}$ is a permutation diagram and $Sp_D$ is multiplicity-free.

One gets a weaker version of Conjecture~\ref{conj:liu} by comparing degrees. The \emph{degree} of a codimension $d$ subvariety $X$ of $\Gr_k(n)$ is the integer $\deg(X)$ defined by $[X]\sigma_1^{k(n-k)-d} = \deg(X)\sigma_{(k^{n-k})}$. Under the Pl\"ucker embedding, this gives the usual notion of the degree of a subvariety of projective space, namely the number of points in the intersection of $X$ with a generic $d$-dimensional linear subspace. One can check using Pieri's rule that $\deg(\sigma_{\lambda}) = f^{\lambda^{\vee}}$, the number of standard Young tableaux of shape $\lambda^\vee$. This is also $\dim Sp_{\lambda^\vee}$. Since degree is additive on cohomology classes, Conjecture~\ref{conj:liu} predicts the following.

\begin{conj}[Liu] \label{conj:liudegree} The degree of $X_{D}$ is $\dim Sp_{D^{\vee}}$. \end{conj}

Liu proved Conjecture~\ref{conj:liudegree} when $D^{\vee}$ is a permutation diagram, and when $D^{\vee}$ has the property that if $(i,j_1), (i,j_2) \in D$ and $j_1 < j < j_2$, then $(i,j) \in D$. In light of the assertion of Theorem \ref{thm:magyar} that taking complements in the decomposition of $Sp_D$ gives the decomposition of $Sp_{D^\vee}$, it is tempting to gloss over the distinction between $D$ and $D^{\vee}$. In fact, the analogue of Theorem \ref{thm:magyar} fails for the classes $[X_D]$, and Conjecture~\ref{conj:liu} can fail for $D$ while holding for $D^{\vee}$.

Suppose $D = \{(1,1), (2,2), (3,3), (4,4)\}$, with $k = 4$ and $n = 8$. This is the skew shape $4321/321$. The Specht module $Sp_D$ is simply the regular representation of $S_4$, with
\begin{equation*}
Sp_D \simeq Sp_{1111} \oplus 3Sp_{211} \oplus 2Sp_{22} \oplus 3Sp_{31} \oplus Sp_{4}.
\end{equation*}
Theorem \ref{thm:magyar} then says
\begin{equation*}
Sp_{D^{\vee}} \simeq Sp_{3333} \oplus 3Sp_{4332} \oplus 2Sp_{4422} \oplus 3Sp_{4431} \oplus Sp_{444},
\end{equation*}
so $\dim Sp_{D^{\vee}} = f^{3333} + 3f^{4332} + 2f^{4422} + 3f^{4431} + f^{444} = 24024$.

On the other hand, an explicit calculation in Macaulay2 shows $\deg X_D = 21384$. Therefore Conjectures~\ref{conj:liudegree} and \ref{conj:liu} both fail for $D$. (One may wonder how such a seemingly small counterexample remained undetected. It is perhaps more natural to index diagram varieties by $D^{\vee}$ than $D$---notice that the cases mentioned above for which Conjecture~\ref{conj:liu} has been established all have the property that $D^{\vee}$, rather than $D$, falls into some nice class of diagrams---and from this point of view the counterexample is no longer so small.)

The discrepancy in degrees is $24024 - 21384 = 2640 = f^{4422}$, which hints at how to see this discrepancy more explicitly. Given a $k$-subset $I$ of $[n]$, write $p_I$ for the corresponding Pl\"ucker coordinate on $\Gr_k(n)$, so $p_I(A)$ is the maximal minor of $A$ in columns $I$. Let $Y$ be the subscheme determined by the vanishing of the Pl\"ucker coordinates $p_{1678}, p_{2578}, p_{3568}, p_{4567}$. These are exactly the Pl\"ucker coordinates which vanish on $X_D$. One can check by computer that $Y$ is a complete intersection, so that $[Y] = \sigma_{1}^4 = \sigma_{1111} + 3\sigma_{211} + 2\sigma_{22} + 3\sigma_{31} + \sigma_4$.

Since the four Pl\"ucker coordinates cutting out $Y$ vanish on $X_D^{\circ}$, the diagram variety $X_D$ is contained in $Y$. However, $Y$ has another component, namely the Schubert variety which is the closure of
\begin{equation*}
\left\{ \rowspan
\begin{bmatrix}
* & * & 1 & 0 & 0 & 0 & 0 & 0\\
* & * & 0 & 1 & 0 & 0 & 0 & 0\\
* & * & 0 & 0 & * & * & 1 & 0\\
* & * & 0 & 0 & * & * & 0 & 1
\end{bmatrix} \right\}.
\end{equation*}
This Schubert variety has degree $\dim Sp_{(22)^{\vee}} = f^{4422} = 2640$, which is $\deg Y - \deg X_D$. Therefore
\begin{equation*}
[X_D] = [Y] - \sigma_{22} = \sigma_{1111} + 3\sigma_{211} + \sigma_{22} + 3\sigma_{31} + \sigma_4.
\end{equation*}

Larger counterexamples to Conjecture~\ref{conj:liu} can be easily manufactured from this one. For two diagrams $D_1$ and $D_2$ where $D_1 \subseteq [a] \times [b]$, define
\begin{equation*}
D_1 \cdot D_2 = D_1 \cup \{(i+a,j+b) : (i,j) \in D_2\}.
\end{equation*}
Graphically, $D_1 \cdot D_2$ is the diagram
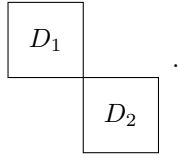
\begin{equation*}
\begin{tikzpicture}
\draw (0,0) rectangle (-1,1);
\draw (0,0) rectangle (1,-1);
\draw (-.5,.5) node {$D_1$};
\draw (.5,-.5) node {$D_2$};
\end{tikzpicture}\,\,\,\raisebox{.45in}{.}
\end{equation*}
One can show that $[X_{D_1 \cdot D_2}] = [X_{D_1}][X_{D_2}]$ and similarly that $s_{D_1 \cdot D_2} = s_{D_1} s_{D_2}$. Therefore if Conjecture~\ref{conj:liu} holds for $D_1$ but not $D_2$, then it will fail for $D_1 \cdot D_2$.

\begin{rem}
It is natural to wonder about the diagram
\begin{equation*}
D' = \{(1,1),(2,2),(3,3),(4,4),(5,5)\},
\end{equation*}
and whether Conjecture~\ref{conj:liu} fails for $D'$. Trying to repeat the analysis above runs into an immediate problem, however (I thank Ricky Liu for pointing this out). Namely, the analogue of $Y$, which is the scheme $Z$ cut out by 
\begin{equation*}
p_{1789(10)}, p_{2689(10)}, p_{3679(10)}, p_{4678(10)}, p_{56789}
\end{equation*}
no longer even has the same codimension as $X_D$. Indeed, $X_D$ has codimension $5$ but $Z$ contains the codimension $4$ Schubert cell
\begin{equation*}
\left\{ \rowspan \begin{bmatrix}
* & * & * & 1 & 0 & 0 & 0 & 0 & 0 & 0\\
* & * & * & 0 & 1 & 0 & 0 & 0 & 0 & 0\\
* & * & * & 0 & 0 & * & * & 1 & 0 & 0\\
* & * & * & 0 & 0 & * & * & 0 & 1 & 0\\
* & * & * & 0 & 0 & * & * & 0 & 0 & 1
\end{bmatrix} \right\}.
\end{equation*}
\end{rem}

\section{Cohomology classes of interval positroid varieties}
\label{sec:positroid-rank-varieties}

\subsection{Positroid varieties}

\begin{defn} An \emph{affine permutation} of quasi-period $n$ is a bijection $f : \ZZ \to \ZZ$ such that $f(i+n) = f(i) + n$ for all $i$. Write $\affS_n$ for the set of affine permutations with quasi-period $n$.
\end{defn}

Note that an $f \in \affS_n$ is completely determined by any sequence $f(a), f(a+1), \ldots, f(a+n-1)$, which we call a \emph{window}. We will usually specify an affine permutation $f \in \affS_n$ by giving the sequence $f(1), \ldots, f(n)$, so that $14825 \in \affS_5$ fixes $1$, sends $3$ to $8$, $7$ to $9$, etc. Members of any window are all distinct modulo $n$, so $\sum_{i=1}^n f(i) \equiv n(n+1)/2 \pmod{n}$. Let $\av(f)$ be the integer $\frac{1}{n}\sum_{i=1}^n (f(i) - i)$.

Write $\affS_n^k$ for the set of affine permutations with $\av(f) = k$. In particular, $\affS_n^0$ is a Coxeter group with simple generators $s_0, \ldots, s_{n-1}$, where $s_i$ interchanges $i+np$ and $i+1+np$ for every $p$. The groups $\affS_n^0$ are the affine Weyl groups of type $A$, and one should beware that affine permutations are frequently defined to be members of $\affS_n^0$ rather than by our broader definition. The shift map $\tau : \ZZ \to \ZZ$, $\tau(i) = i+1$ yields a bijection $\affS_n^0 \to \affS_n^k$ for each $k$, namely $f \mapsto \tau^k f$, and we will use these bijections to transport Coxeter structure from $\affS_n^0$ to any $\affS_n^k$. For instance, we define the reduced words of $f \in \affS_n^k$ to be the reduced words of $\tau^{-k}f \in \affS_n^0$. The next definition provides another example.

\begin{defn} The \emph{length} $\ell(f)$ of an affine permutation $f$ is the number of inversions $i < j$, $f(i) > f(j)$, provided that we regard any two inversions $i < j$ and $i + pn < j + pn$ as equivalent. \end{defn}
Clearly $\ell(\tau f) = \ell(f)$, and one checks that $\ell(f)$ agrees with the usual Coxeter length when $f \in \affS_n^0$.

\begin{defn} An affine permutation $f \in \affS_n$ is \emph{bounded} if $i \leq f(i) \leq i+n$ for all $i$. Let $\Bound(k,n)$ denote the set of bounded affine permutations in $\affS_n^k$. \end{defn}

The next proposition makes it easy to identify members of $\Bound(k,n)$.
\begin{prop} An affine permutation $f$ is in $\Bound(k,n)$ if and only if it is bounded and exactly $k$ of $f(1), \ldots, f(n)$ exceed $n$. \end{prop}

Any affine permutation $f$ has a permutation matrix, the $\ZZ \times \ZZ$ matrix $A$ with $A_{i,f(i)} = 1$ and all other entries $0$. For any $i,j \in \ZZ$, define
\begin{equation*}
[i,j](f) = \{p < i : f(p) > j\}. \label{eq:perm-diagram-northeast-count}
\end{equation*}
That is, $\#[i,j](f)$ is the number of $1$'s strictly northeast of $(i,j)$ in the permutation matrix of $f$, in matrix coordinates.

Fix a basis $e_1, \ldots, e_n$ of $\CC^n$. With this choice in mind, we adopt the following abuse of notation: if $X \subseteq \CC^n$, $\langle X \rangle$ will mean the span of $X$, while if $X \subseteq [n]$, $\langle X \rangle$ will mean the span of $\{e_i : i \in X\}$. For $X \subseteq [n]$, let $\Prj_X : \CC^n \to \langle X \rangle$ be the projection which fixes those basis vectors $e_i$ with $i \in X$ and sends the rest to $0$. For integers $i \leq j$, write $[i, j]$ for $\{i, i+1, \ldots, j\}$. We interpret indices of basis vectors modulo $n$, so that $\langle [i,j] \rangle \subseteq \CC^n$ even if $i, j$ fail to lie in $[1,n]$.

\begin{defn}[\cite{positroidjuggling}] \label{defn:positroid-variety}
Given a bounded affine permutation $f \in \Bound(k,n)$, the \emph{positroid variety} $\Pi_f \subseteq \Gr_k(n)$ is
\begin{equation*}
\{V \in \Gr_k(n) : \dim \Prj_{[i,j]} V \leq k - \#[i,j](f) \text{ for all $i \leq j$} \}.
\end{equation*}
\end{defn}

\begin{thm}[\cite{positroidjuggling}, Theorem 5.9] \label{thm:positroid-variety}
The positroid variety $\Pi_f \subseteq \Gr_k(n)$ is irreducible of codimension $\ell(f)$. \end{thm}

Knutson--Lam--Speyer also computed the cohomology class of $\Pi_f$ in terms of affine Stanley symmetric functions. These are a class of symmetric functions indexed by affine permutations introduced by Lam in \cite{lam-affine-stanley}, which we now define.

A \emph{reduced word} for $f \in \affS_n^0$ is a word $a_1 \cdots a_{\ell}$ in the alphabet $[0,n-1]$ with $s_{a_1} \cdots s_{a_{\ell}} = f$ and such that $\ell$ is minimal with this property. Let $\Red(f)$ denote the set of reduced words for $f$. A reduced word $a = a_1 \cdots a_{\ell}$ is \emph{cyclically decreasing} if all the $a_i$ are distinct, and if whenever some $j$ and $j+1$ appear in $a$ (modulo $n$), $j+1$ precedes $j$. An affine permutation is cyclically decreasing if it has a cyclically decreasing reduced word. For a partition $\lambda$, let $m_{\lambda}$ be the monomial symmetric function indexed by $\lambda$.

\begin{defn} \label{defn:affine-stanley}
The \emph{affine Stanley symmetric function} of $f \in \affS_n^0$ is
\begin{equation*}
\tilde{F}_f = \sum_{(f^1, \ldots, f^p)} x_1^{\ell(f^1)} \cdots x_p^{\ell(f^p)},
\end{equation*}
where $(f^1, \ldots, f^p)$ runs over all factorizations $f = f^1 \cdots f^{p}$ with each $f_i$ cyclically decreasing.
\end{defn}
As above, we extend this definition to $f \in \affS_n^k$ for arbitrary $k$ by defining $\tilde{F}_f$ as $\tilde{F}_{\tau^{-k}f}$.

\begin{thm}[\cite{positroidjuggling}, Theorem 7.1] \label{thm:positroid-variety-class}
For $f \in \Bound(k,n)$, the cohomology class $[\Pi_f]$ is $\phi(\tilde{F}_f)$.
\end{thm}

The ordinary Stanley symmetric functions indexed by members of $S_n$, introduced by Stanley in \cite{stanleysymm}, are examples of affine Stanley symmetric functions. To be precise, we can view $w \in S_n$ as the affine permutation in $\affS^0_n$ sending $i+pn$ to $w(i)+pn$ for $1 \leq i \leq n$. Then the Stanley symmetric function $F_w$ of $w$ is $\tilde{F}_w$. This is Proposition 5 in \cite{lam-affine-stanley}, but we will simply take it as a definition of $F_w$. One should be aware, however, that the $F_w$ defined in \cite{stanleysymm} is our $F_{w^{-1}}$.

\subsection{Grassmann duality}
Let $\Gr^k(n)$ be the Grassmannian of $k$-planes in $(\CC^n)^*$. The \emph{annihilator} of $V \in \Gr_k(n)$ is
\begin{equation*}
\ann(V) = \{\alpha \in (\CC^n)^* : \alpha|_V = 0\} \in \Gr^{n-k}(n).
\end{equation*}
The map $\Gr_k(n) \to \Gr^{n-k}(n)$ sending $V$ to $\ann(V)$ is an isomorphism, and we will refer to a pair of closed subvarieties which correspond under this isomorphism as \emph{Grassmann duals}.

Let $\epsilon_1, \ldots, \epsilon_n$ denote the dual basis of $e_1, \ldots, e_n$. For $S \subseteq [n]$, we write $\bar{S}$ for $[n] \setminus S$ and $\langle S^* \rangle$ for $\langle \epsilon_i : i \in S \rangle$. Observe that if $f \in \Bound(k,n)$, then $\tau^n f^{-1} \in \Bound(n{-}k,n)$.

\begin{lem}[\cite{knutson-interval-positroid-varieties}, Proposition 2.1] \label{lem:grassmann-duality} For $f \in \Bound(k,n)$, the positroid varieties $\Pi_f \subseteq \Gr_k(n)$ and $\Pi_{\tau^n f^{-1}} \subseteq \Gr^{n-k}(n)$ are Grassmann dual. \end{lem}

Lemma~\ref{lem:grassmann-duality} is straightforward given the following technical lemma, which will also be useful later on.
\begin{lem} \label{lem:NE-vs-SW} For $f \in \Bound(k,n)$ and $i \leq j \leq i+n$, let $a = \#[i,j](f)$ and $b$ be the number of $1$'s in the permutation matrix of $f$ which are strictly northeast and weakly southwest of $(i,j)$, respectively. Then $\#[i,j] + a = k + b$. \end{lem}
\begin{proof}
Consider the following part of the permutation matrix of $f$, divided into four regions:
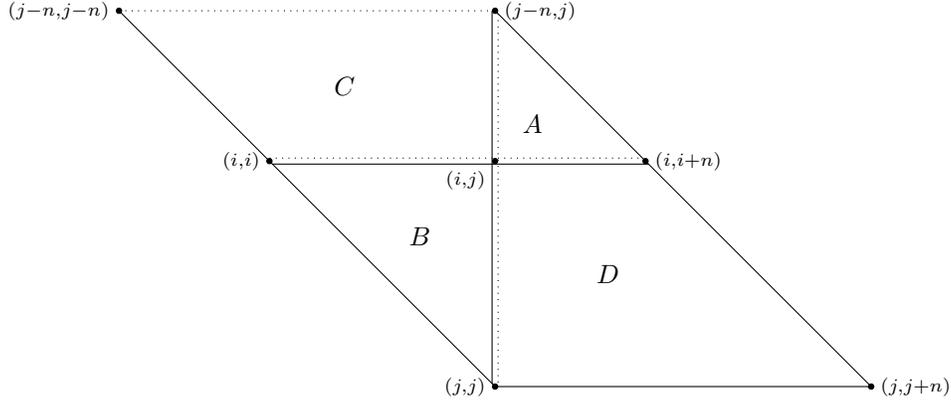
\begin{center}
\begin{tikzpicture}
\draw (0,0) -- (5,-5);
\draw (5,0) -- (10,-5);

\draw[dotted] (0,0) -- (5,0);
\draw (5,-5) -- (10,-5);

\draw[dotted] (2,-1.96) -- (6.96,-1.96);
\draw (2.04,-2.04) -- (7.04,-2.04);

\draw (4.96,0) -- (4.96,-4.96);
\draw[dotted] (5.04,-.04) -- (5.04,-5);

\filldraw (5,-2) circle (1pt) node[below left] {$\scriptstyle (i,j)$};
\filldraw (2,-2) circle (1pt) node[left] {$\scriptstyle (i,i)$};
\filldraw (7,-2) circle (1pt) node[right] {$\scriptstyle (i,i+n)$};
\filldraw (5,0) circle (1pt) node[right] {$\scriptstyle (j-n,j)$};
\filldraw (0,0) circle (1pt) node[left] {$\scriptstyle (j-n,j-n)$};
\filldraw (5,-5) circle (1pt) node[left] {$\scriptstyle (j,j)$};
\filldraw (10,-5) circle (1pt) node[right] {$\scriptstyle (j,j+n)$};

\draw (3, -1) node {$C$};
\draw (5.5,-1.5) node {$A$};
\draw (4,-3) node {$B$};
\draw (6.5,-3.5) node {$D$};
\end{tikzpicture}
\end{center}
Here a line segment on the boundary of a region is included in the region if the segment is solid, and not included if it is dotted. For instance, $C = \{(p,q) : j{-}n < p < i, p \leq q \leq j\}$. Let $a, b, c, d$ denote the number of $1$'s in the regions $A, B, C, D$. Boundedness of $f$ implies that all the $1$'s in its permutation matrix lie (weakly) between the two diagonal lines in this picture, so since $B \cup D$ contains $\#[i,j]$ rows we have $b+d = \#[i,j]$. Since $f \in \Bound(k,n)$, exactly $k$ of $f(1), \ldots, f(n)$ exceed $n$, and by quasi-periodicity this says $a+d = k$.  But now $\#[i,j] + a = b+d+a = k+b$.
\end{proof}

\begin{proof}[Proof of Lemma~\ref{lem:grassmann-duality}]
Take $V \in \Gr_k(n)$. We claim that for any cyclic interval $[i,j]$ in $[n]$,
\begin{equation*}
    \dim \Prj_{[i,j]} V \leq k-\#[i,j](f) \quad \Longleftrightarrow \quad \dim \Prj_{\overline{[i,j]}^*} \ann(V) \leq (n-k)-\#\overline{[i,j]}(\tau^n f^{-1}),
\end{equation*}
which will prove the lemma according to Definition~\ref{defn:positroid-variety}. For any $S \subseteq [n]$, the rank of the composite $V \hookrightarrow \CC^n \twoheadrightarrow \CC^n / \langle \bar{S} \rangle$ is $\dim \Prj_S V$, and by dualizing one sees that this is the same as $\#S - (n{-}k) + \dim \Prj_{\bar{S}^*} \ann(V)$. Taking $S = [i,j]$,
\begin{equation*}
    \dim \Prj_{[i,j]} V \leq k-\#[i,j](f) \quad \Longleftrightarrow \quad \dim \Prj_{\overline{[i,j]}^*} \ann(V) \leq n-\#[i,j]-\#[i,j](f).
\end{equation*}
Thus, to prove the claim we must show that
\begin{equation} \label{eq:grassmann-duality}
\#[i,j] + \#[i,j](f) = k + \#\overline{[i,j]}(\tau^n f^{-1}).
\end{equation}
The permutation matrix of $f$ is the permutation matrix of $\tau^n f^{-1}$ shifted left $n$ units and reflected across the diagonal of $\ZZ \times \ZZ$, and so $\#\overline{[i,j]}(\tau^n f^{-1}) = \#[j{+}1,n{+}i{-}1](\tau^n f^{-1})$ is the number of $1$'s weakly southwest of $(i,j)$ in the permutation matrix of $f$.  Lemma~\ref{lem:NE-vs-SW} now implies equation \eqref{eq:grassmann-duality}.
\end{proof}

\subsection{Interval positroid varieties} \label{subsec:interval-positroid}

An \emph{interval positroid variety} is one for which all rank conditions in Definition~\ref{defn:positroid-variety} are implied by conditions involving actual intervals in $[n]$.

\begin{thm}[\cite{knutson-interval-positroid-varieties}] \label{thm:interval-positroid-classification} For $f \in \Bound(k,n)$, $\Pi_f$ is an interval positroid variety if and only if the subsequence of $f(1), \ldots, f(n)$ consisting of the entries exceeding $n$ is increasing. 
\end{thm}
Any $f$ as in the preceding theorem is determined by the subsequence of $f(1), \ldots, f(n)$ of entries not exceeding $n$, which is a \emph{partial permutation}, i.e. an injection from a subset of $[n]$ into $[n]$. Let $\bar{f}$ denote the partial permutation associated to $f \in \Bound(k,n)$ in this way. For instance, if $f = 15748$ then $\bar{f} = 15\mathunderscore 4\mathunderscore$, where a $\mathunderscore$ in position $i$ indicates that $i$ is not in the domain of $\bar{f}$. Conversely, if the domain $\dom(\bar{f})$ has size $n-k$ and $\bar{f}(i) \geq i$ for $i \in \dom(\bar f)$, then $\bar{f}$ labels an interval positroid variety.  We now describe a different way to index interval positroid varieties, following \cite{rank-varieties} (up to Grassmann duality).
\begin{defn}[\cite{rank-varieties}] A \emph{rank set} in $[n]$ is a finite set of intervals $M = \{[a_1,b_1], \ldots, [a_m,b_m]\}$ with $a_i \leq b_i \leq n$ positive integers, where all $a_i$ are distinct and all $b_i$ are distinct. For $S \subseteq [n]$, let $S(M)$ denote the set of intervals $S' \in M$ such that $S' \subseteq S$.\end{defn}

To a rank set $M$ in $[n]$ with $n-k$ intervals we associate the variety
\begin{equation*}
\Pi_M = \{V \in \Gr_k(n) : \text{$\dim \Prj_S V \leq \#S - \#S(M)$ for all intervals $S \subseteq [n]$}\}.
\end{equation*}
This is in fact an interval positroid variety, labelled by the affine permutation constructed as follows. Say $M = \{[a_1, b_1], \ldots, [a_{n-k}, b_{n-k}]\}$ is a rank set with $a_1 < \cdots < a_{n-k} \leq n$. Define
\begin{equation*}
\{c_1 < \cdots < c_{k}\} = [n] \setminus \{a_1, \ldots, a_{n-k}\} \quad \text{ and} \quad \{d_1 < \cdots < d_{k}\} = [n] \setminus \{b_1, \ldots, b_{n-k}\}.
\end{equation*}
Let $f_M \in \affS_n$ be the affine permutation which maps $a_i$ to $b_i$ and $c_i$ to $d_i+n$. Then $f_M$ is bounded because $a_i \leq b_i$, which implies $d_i \leq c_i$.

\begin{ex}
Take $M = \{[1,1],[2,5],[4,4]\}$ and $n = 5$. Then $f_M = 15748$ and $\bar{f}_M = 15\mathunderscore 4\mathunderscore$.
\end{ex}

\begin{lem} \label{lem:rank-positroid-correspondence} For a rank set $M$ in $[n]$ we have $\Pi_M = \Pi_{f_M}$. \end{lem}
\begin{proof}
By construction, the entries of $f_M(1), \ldots, f_M(n)$ exceeding $n$ appear in increasing order, so $\Pi_M$ is an interval positroid variety by Theorem~\ref{thm:interval-positroid-classification}. Therefore it suffices to show that $\#[i,j] - \#[i,j](M) = k - \#[i,j](f_M)$ for all intervals $[i,j]$ in $[n]$.

Let $B = \{q \in \ZZ : \text{$(q, f_M(q))$ is weakly southwest of $(i,j)$}\}$. We claim that $\#[i,j](M) = \#B$, in which case we are done by Lemma~\ref{lem:NE-vs-SW}.  Clearly $[a_p, b_p] = [a_p, f_M(a_p)] \subseteq [i,j]$ if and only if $a_p \in B$, so $\#[i,j](M) = B \cap \{a_1, \ldots, a_{n-k}\}$. But in fact every $q \in B$ is some $a_p$, because $f_M(q) \leq j \leq n$ and $1 \leq i \leq q$ force $q \in \{a_1, \ldots, a_{n-k}\}$.
\end{proof}

It follows from Theorem~\ref{thm:positroid-variety} that $\Pi_M$ is irreducible of dimension $k(n-k) - \ell(f_M)$. The next lemma gives a formula for this dimension more directly in terms of $M$ (cf. \cite[Lemma 3.29]{two-step-flags}).
\begin{lem} \label{lem:dimensions} For any rank set $M$, $\dim \Pi_M = \sum_{S \in M} (\#S - \#S(M))$. \end{lem}

\begin{proof}
As before, write $M = \{[a_1, b_1], \ldots, [a_{n-k}, b_{n-k}]\}$ where $a_1 < \cdots < a_{n-k}$. Also, write $\dim(M)$ for $\sum_{S \in M} (\#S - \#S(M))$, so we want to prove that $\dim(M) = \dim \Pi_M$. Let $i(M)$ be the maximal $i \in [n-k]$ such that $a_i < k+i$; if no such $i$ exists, set $i(M) = -\infty$. When $i(M)$ is finite, we will define a new rank set $M'$ with the property that either $\dim(M') < \dim(M)$, or $\dim(M') = \dim(M)$ and $i(M') < i(M)$. Thus, after finitely many operations of the form $M \mapsto M'$ we obtain an $M''$ with $i(M'') = -\infty$, which must be $M'' = \{\{k+1\}, \{k+2\}, \ldots, \{n\}\}$. In this case $f_{M''} = (n+1)\cdots(n+k)(k+1)\cdots n$ has length $k(n-k)$, so $\dim \Pi_{M''} = 0$ and the lemma holds. It therefore suffices to show that $\dim(M)-\dim(M') = \dim \Pi_M - \dim \Pi_{M'}$.
\begin{enumerate}[(a)]
\item First suppose $a_i < b_i$.  Let $M'$ be $M$ with $S = [a_i, b_i]$ replaced by $S' = [a_i+1, b_i]$. The choice of $i$ implies that $a_i + 1$ remains in $[n]$ and is not the left endpoint of an interval in $M$, so $M'$ is a valid rank set. Moreover, the multiset of numbers $\#T(M)$ for $T \in M$ is the same as the multiset of numbers $\#T'(M')$ for $T' \in M'$, so $\dim(M)-\dim(M') = 1$. On the other hand, $f_M$ and $f_{M'}$ agree except in positions $a_{i}$ and $a_{i}+1$, where
\begin{equation*}
    \begin{array}{ll}
        f_M(a_{i}) = b_{i}, & f_M(a_{i}+1) = d_j+n \text{ (for some $j$)}\\
        f_{M'}(a_{i}) = d_j+n, & f_{M'}(a_{i}+1) = b_{i}.
    \end{array}
\end{equation*}
In particular, $f_{M'} = f_M s_{a_{i}} > f_M$ in weak Bruhat order, so
\begin{equation*}
    \dim \Pi_M - \dim \Pi_{M'} = \ell(f_{M'}) - \ell(f_M) = 1 = \dim(M) - \dim(M').
\end{equation*}

\item Now suppose $a_i = b_i$.
\begin{enumerate}[(i)]
    \item Suppose $a_i+1$ is not the \emph{right} endpoint of an interval. Define $M'$ to be $M$ with $[a_i, a_i]$ replaced by $[a_i+1,a_i+1]$. Then $M'$ is a valid rank set with $\dim M' = \dim M$, On the other hand, $\Pi_{M'}$ is the image of $\Pi_M$ under the invertible linear map switching $e_{a_i}$ with $e_{a_i+1}$ and fixing all other $e_j$, and so $\dim \Pi_{M'} = \dim \Pi_M$.
    
    \item Suppose $a_i+1 = b_h$ for some $h$. Define $M'$ to be $M$ with $[a_i, a_i]$ replaced by $[a_i+1, a_i+1]$ and $[a_h, b_h]$ replaced by $[a_h, b_h-1] = [a_h, a_i]$. This is a valid rank set, and one checks that $\dim(M) = \dim(M')$ again. The affine permutations $f_M$ and $f_{M'}$ agree except that
    \begin{equation*}
        \begin{array}{lll}
            f_M(a_h) = a_i+1, & f_M(a_i) = a_i, & f_M(a_i+1) = d_j+n \text{ (for some $j$)}\\
            f_{M'}(a_h) = a_i, & f_{M'}(a_{i}) = d_j+n, &f_{M'}(a_i+1) = a_i+1
        \end{array}
    \end{equation*}
    Hence, $f_{M'} = s_{a_i} f_M s_{a_i}$ with $f_M < f_M s_{a_i} > s_{a_i} f_M s_{a_i}$ in weak Bruhat order. In particular, $\ell(f_M) = \ell(f_{M'})$ so that $\dim \Pi_M = \dim \Pi_{M'}$.
\end{enumerate}
In either case, $\dim(M) = \dim(M')$ and $\dim \Pi_M = \dim \Pi_{M'}$. If $a_i+1 < k+i$, then $i(M') = i(M)$, but after $k+i-a_i$ steps of type (b) the statistic $i$ must decrease.
\end{enumerate}
\end{proof}

\subsection{Stability}

Fix inclusions $\CC \subseteq \CC^2 \subseteq \cdots$ and linearly independent vectors $e_1, e_2, \ldots$ with $e_i \in \CC^i$ for all $i$. Let $R_{k,n}$ denote the homogeneous coordinate ring of $\Gr_k(n)$ under the Pl\"ucker embedding, so $R_{k,n}$ is generated by Pl\"ucker coordinates $p_I$ for $I \in {[n] \choose k}$. Any Pl\"ucker relation in $R_{k,n}$ is still a Pl\"ucker relation in $R_{k,n+1}$, so there are injective ring homomorphisms $R_{k,n} \hookrightarrow R_{k,n+1} \hookrightarrow \cdots$ sending $p_I$ to $p_I$, which we view as inclusions. Given a subscheme $Z \subseteq \Gr_k(n)$ determined by a homogeneous ideal $J \subseteq R_{k,n}$, let $Z^+$ be the subscheme of $\Gr_k(n+1)$ determined by the ideal $R_{k,n+1}J$. That is, $Z^+$ is cut out by the same equations as $Z$, but now inside $\Gr_{k}(n+1)$.

\begin{prop} \label{prop:pullback} Let $\iota : \Gr_k(n) \to \Gr_k(n{+}1)$ be the inclusion, inducing a pullback $\iota^* : H^* \Gr_k(n{+}1) \to H^* \Gr_k(n)$. Then $\iota^*[Z^+] = [Z]$. \end{prop}

\begin{proof} Whenever $Y \subseteq \Gr_k(n+1)$ intersects $\iota \Gr_k(n)$ transversely it holds that $\iota^* [Y] = [Y \cap\iota \Gr_k(n)]$ with $[Y \cap \iota\Gr_k(n)]$ viewed as a cycle on $\Gr_k(n)$, and one can verify that $Z^+$ intersects $\iota \Gr_k(n)$ transversely by working in charts. \end{proof}

Let $\Lambda_k$ be the ring of symmetric polynomials over $\ZZ$ in $x_1, \ldots, x_k$. Then $H^* \Gr_k(n) \simeq \Lambda_k / (s_{\lambda} : \lambda \not\subseteq [k] \times [n-k])$, and these isomorphisms induce an isomorphism of the inverse limit $\displaystyle \varprojlim_N H^* \Gr_k(N)$ with $\Lambda_k$. Here, we take the inverse limit with respect to the maps
\begin{equation*}
\cdots \xrightarrow{\iota^*} H^*\Gr_k(k+1) \xrightarrow{\iota^*} H^*\Gr_k(k).
\end{equation*}
Proposition~\ref{prop:pullback} shows that the classes $[Z], [Z^+], [Z^{++}], \ldots$ define an element $\alpha \in \varprojlim H^* \Gr_k(N)$, and we say $F \in \Lambda_k$ is a \emph{stable} representative for $[Z]$ if it represents $\alpha$.

Now suppose $M$ is a rank set for $\Gr_k(n)$. Define $M^+$ to be $M \cup \{[a,n+1]\}$ where $a$ is the minimal member of $[n+1]$ which is not a left endpoint in $[n]$. Evidently $M^+$ is a rank set for $\Gr_k(n+1)$.
\begin{lem} \label{lem:stable-interval-positroid} $\Pi_{M^+} = \Pi_M^+$. \end{lem}
\begin{proof}
Let $S \subseteq [n+1]$ be an interval, and consider a rank condition
\begin{equation} \label{eq:rank-condition}
\dim \Prj_S V \leq \#S - \#S(M^+)
\end{equation}
for $\Pi_{M^+}$. We must see that \eqref{eq:rank-condition} follows from the rank conditions defining $\Pi_M$. Consider three cases.
\begin{enumerate}[(a)]
\item If $n+1 \notin S$, then $S(M^+) = S(M)$, and \eqref{eq:rank-condition} is itself a rank condition defining $\Pi_M$.
\item Suppose $S = [i,n+1]$ with $i > a$, and set $S' = [i,n]$. Then $\#S - \#S(M^+) = \#S' - \#S'(M) + 1$, so \eqref{eq:rank-condition} follows from the rank condition $\dim \Prj_{S'} V \leq \#S' - \#S'(M)$ for $\Pi_M$.
\item Suppose $S = [i,n+1]$ with $i \leq a$. Then $S$ contains every interval of $M^+$ except $[1, b_1], \ldots, [i-1, b_{i-1}]$, and so $\#S - \#S(M^+) = \#[i,n+1] - (\#M^+ - (i-1)) = k$: the rank condition \eqref{eq:rank-condition} is vacuous.
\end{enumerate}
\end{proof}

Let $M^{+r}$ denote the result of applying the $^+$ operation $r$ times starting with $M$; when $f = f_M$, we also write $f^{+r}$ and $\bar f^{+r}$ to mean $f_{M^{+r}}$ and $\bar f_{M^{+r}}$. Write $S_{\infty}$ for the union $\bigcup_{n=0}^{\infty} S_n$, identifying $S_n$ with the subgroup of $S_{n+1}$ fixing $n+1$.
\begin{lem} \label{lem:stable-rank-sets} Let $M$ be a rank set for $\Gr_k(n)$. There exists an integer $R$ such that
\begin{itemize}
\item $f_M^{+r}\tau^{-k} \in S_{n+r}$ for $r \geq R$, and
\item the permutations $f_M^{+r}\tau^{-k}$ for $r \geq R$ are all the same as members of $S_{\infty}$.
\end{itemize}
\end{lem}
\begin{proof} Suppose first $\bar f_M$ has domain $[1,n-k]$, so $\bar f_M = b_1 \cdots b_{n-k} \mathunderscore \ldots \mathunderscore$. Then $f_M \tau^{-k} = d_1 \cdots d_k b_1 \cdots b_{n-k}$ is in $S_n$. In general, $\bar f^+$ is the partial permutation of $[n+1]$ agreeing with $\bar f$ on $\dom(\bar f)$, and sending the minimal member of $[n+1] \setminus \dom(\bar f)$ to $n+1$. Thus, $f_M^+ \tau^{-k} = d_1 \cdots d_k b_1 \cdots b_{n-k} (n+1)$, which is equal to $f_M \tau^{-k}$ as a member of $S_{\infty}$.

For an arbitrary $\bar f_M$, it suffices by the previous paragraph to find $R$ such that $\bar f_M^{+R}$ has domain $[1,n+R-k]$. Any $R$ such that $\dom(\bar f_M) \subseteq [1, R + \#\dom(\bar f_M)]$ does the job.
\end{proof}

\begin{thm} \label{thm:interval-positroid-class} For any interval positroid variety $\Pi_M$, there is an ordinary permutation $w$ such that the Stanley symmetric function $F_w$ is a stable representative for the class $[\Pi_M]$. \end{thm}

\begin{proof} Since the reduced words of a permutation $w$ only depend on $w$ as an element of $S_{\infty}$, the same is true of $F_w$. Lemma~\ref{lem:stable-rank-sets} therefore shows that the sequence $\tilde{F}_{f_M^{+r}}$ for $r \geq 0$ is eventually constant and equal to some $F_w$. These symmetric functions represent the classes $[\Pi_M^{+r}]$ by Lemma~\ref{lem:stable-interval-positroid}, so $F_w$ stably represents the class $[\Pi_M]$.
\end{proof}

Although $\phi(\tilde{F}_f)$ must be Schubert-positive, and it is known that $F_w$ is Schur-positive \cite{edelman-greene}, the symmetric functions $\tilde{F}_f$ are \emph{not} always Schur-positive. For instance, if $M = \{[2,2], [4,4]\}$ with $\Sigma_M \subseteq \Gr_2(4)$, then $f_M = 5274$, and $\tilde{F}_{5274} = s_{22} + s_{211} - s_{1111}$. On the other hand, $M^{++} = \{[2,2],[4,4],[1,5],[3,6]\}$, $f_M^+ = 526479$, and $\tilde{F}_{526479} = F_{135264} = s_{22} + s_{211}$. Thus, Theorem~\ref{thm:interval-positroid-class} provides a canonical way to represent interval positroid classes by Schur-positive symmetric functions.

\section{Degenerations of dual interval positroid varieties}

Given a subset $E \subseteq [k] \times [n]$, define
\begin{equation*}
\Sigma^\circ_E = \{\rowspan A : A \in M_{k,n} \text{ such that $A_{pq} = 0$ whenever $(p,q) \notin E$}\} \subseteq \Gr_k(n)
\end{equation*}
and $\Sigma_E = \overline{\Sigma^\circ_E}$. For a generic $V = \rowspan A \in \Sigma^\circ_E$, the matroid of $V$ is the \emph{transversal matroid} associated to the columns of $E$: that is, the matroid on $[n]$ whose bases are the sets $\{j_1, \ldots, j_k\}$ for which $(1,j_1), \ldots, (k,j_k) \in E$. Thus, $\Sigma_E$ is the closure of a matroid stratum.

We identify a rank set (or any collection of intervals) $M = \{S_1, \ldots, S_k\}$ in $[n]$ with the subset
\begin{equation*}
\{(i,j) : i \in [k], j \in S_i\} \subseteq [k] \times [n],
\end{equation*}
and define $\Sigma_M$ accordingly. For instance, if $M = \{[1,3], [3,6], [4,5]\}$ and $n = 6$, then $\Sigma_M^\circ$ is the set of rowspans of full rank matrices of the form
\begin{equation*}
\begin{bmatrix}
* & * & * & 0 & 0 & 0\\
0 & 0 & * & * & * & *\\
0 & 0 & 0 & * & * & 0
\end{bmatrix}.
\end{equation*}
The varieties $\Sigma_M$ are the ``rank varieties'' defined in \cite{rank-varieties}, where it is shown that they are exactly the projections of Schubert varieties in partial flag varieties $\Fl(k_1, \ldots, k_p; \CC^n)$ with $k_p = k$ to $\Gr_k(n)$ (see also \cite{two-step-flags}).

\begin{lem} \label{lem:rank-variety-dimension} $\dim \Sigma_M = \sum_{S \in M} (\#S - \#S(M))$ for a rank set $M$. \end{lem}
\begin{proof}
    Write $M = \{[a_1, b_1], \ldots, [a_k, b_k]\}$ where $a_1 < \cdots < a_k$. Let $V$ be the set of $k \times (n-k)$ matrices $A$ such that
    \begin{itemize}
        \item $A_{i,a_i} = 1$ for each $i$;
        \item If $j \notin [a_i, b_i]$, then $A_{ij} = 0$;
        \item If $[a_\ell,b_\ell] \subseteq [a_i,b_i]$ with $\ell \neq i$, then $A_{i,a_\ell} = 0$;
        \item If $A_{ij}$ has not been defined already, it is nonzero.
    \end{itemize}
    For example, if $M = \{[1,4],[2,6],[4,5]\}$, then
    \begin{equation*}
    V = \left\{ \begin{bmatrix} 1 & * & * & * & 0 & 0 \\ 0 & 1 & * & 0 & * & * \\ 0 & 0 & 0 & 1 & * & 0 \end{bmatrix} : \text{all $*$ nonzero} \right\}
    \end{equation*}
    Note that $\dim V = \sum_{S \in M} (\#S - \#S(M))$. The map $A \mapsto \rowspan(A)$ takes $V$ onto a dense subset of $\Sigma_M^\circ$, so to prove the lemma it suffices to show that this map is injective, i.e. that if $A, gA \in V$ for some $g \in \GL_k(\CC)$, then $g = 1$.

    Use the Bruhat decomposition of $\GL_k$ to write $g = u_1 t u_2$, where $t$ is diagonal and $u_1$, $u_2$ are respectively upper and lower triangular with $1$'s on the diagonal. If $gA \in V$, then $u_2 = 1$, because otherwise $gA$ would have a nonzero entry below some position $(i,a_i)$. Next, $t = 1$, because otherwise $gA$ would have an entry other than $1$ in some position $(i,a_i)$. Finally, $u_1 = 1$, for otherwise if $u_1$ added a multiple of some row $\ell$ to a row $i < \ell$, then $gA$ would have a nonzero entry in position $A_{i,a_{\ell}}$ (if $b_{\ell} \leq b_i$) or position $A_{i,a_i+1}$ (if $b_{\ell} > b_i$).
\end{proof}
We will not need this fact, but it is worth noting that the proof of Lemma~\ref{lem:rank-variety-dimension} only requires that all left endpoints of intervals in $M$ are distinct, or that all right endpoints are, but not both (as required by the definition of a rank set).

\begin{lem} \label{lem:rank-vs-positroid} The Grassmann dual to an interval positroid variety $\Pi_M \subseteq \Gr^{n-k}(n)$ is $\Sigma_M \subseteq \Gr_k(n)$. \end{lem}
\begin{proof}
Let $\Pi_M^*$ denote the Grassmann dual of $\Pi_M$. Recall that $V \in \Pi_M$ if and only if $\dim \Prj_{S^*} V \leq \#S - \#S(M)$ for all $S \in M$. As in the proof of Lemma~\ref{lem:grassmann-duality},
\begin{equation*}
\dim \Prj_{S^*} V = \#S - k + \dim \Prj_{\bar{S}} \ann(V) = \#S - \dim(\ann(V) \cap \langle S \rangle).
\end{equation*}
Thus, $W = \ann(V) \in \Pi_M^*$ if and only if $\dim(W \cap \langle S \rangle) \geq \#S(M)$ for $S \in M$. These rank conditions hold when $W \in \Sigma^{\circ}_M$, so $\Sigma_M \subseteq \Pi_M^*$. Since $\Pi_M^*$ is irreducible and has the same dimension as $\Sigma_M$ by Lemmas~\ref{lem:dimensions} and \ref{lem:rank-variety-dimension}, we are done.
\end{proof}

Let $\phi_{t,i\to j}$ be the linear transformation sending $e_i$ to $te_i + (1-t)e_j$. For $t \neq 0$, the varieties $\phi_{t,i\to j} \Sigma_M$ are all isomorphic, so they form a flat family \cite[Proposition III-56]{geometry-of-schemes}. The flat limit $\lim_{t \to 0} \phi_{t,i\to j} \Sigma_M$ then exists as a scheme \cite[Proposition 9.8]{hartshorne}. The key fact for us is that $\Sigma_M$ and $\lim_{t \to 0} \phi_{t,i\to j} \Sigma_M$ have the same Chow ring class, hence the same cohomology class. Other authors have used these degenerations to calculate cohomology classes or K-theory classes of subvarieties of Grassmannians, including Coskun \cite{two-step-flags} and Vakil \cite{geometric-littlewood-richardson}. Our goal in this section is to exhibit a degeneration of $\Sigma_M$, for an appropriate $M$, which contains the diagram variety $X_{D(w)}$ as an irreducible component.

For a closed subscheme $X \subseteq \Gr_k(n)$, let $C_{i\to j}X = \lim_{t\to 0} \phi_{t,i\to j} X$. For $E \subseteq [k] \times [n]$, let $C_{i\to j}E$ be the subset of $[k] \times [n]$ obtained from $E$ by replacing columns $i$ and $j$ by their intersection and union, respectively. That is, $(p,q) \in C_{i\to j} E$ if and only if 
\begin{itemize}
\item $q \notin \{i,j\}$ and $(p,q) \in E$, or
\item $q = i$ and $(p,i), (p,j) \in E$, or
\item $q = j$ and $(p,i) \in E$ or $(p,j) \in E$.
\end{itemize}

\begin{lem}[\cite{liuthesis}, Proposition 5.3.3] \label{lem:diagram-degeneration} For any $E \subseteq [k] \times [n]$ we have $\Sigma_{C_{i\to j} E} \subseteq C_{i\to j} \Sigma_E$. \end{lem}

Given a permutation $w \in S_n$, define a rank set $M(w) = \{[w(i), i+n] : 1 \leq i \leq n\}$, so $\Sigma_{M(w)} \subseteq \Gr_n(2n)$. Then
\begin{equation*}
\tau^{2n} f_{M(w)}^{-1} = (n+1)\cdots(2n)(w(1)+2n)\cdots(w(n)+2n) = (w \times 12\cdots n) \tau^{-n}.
\end{equation*}
Here, for $w \in S_n$ and $v \in S_m$, $w \times v$ is the permutation in $S_{n+m}$ sending $i$ to $w(i)$ if $i \leq n$ and to $v(i-n)+n$ otherwise. By Lemmas~\ref{lem:rank-vs-positroid} and \ref{lem:grassmann-duality}, $\Sigma_{M(w)} = \Pi_{\tau^{2n} f_{M(w)}^{-1}}$. It is clear from Definition~\ref{defn:affine-stanley} that $F_{w \times 12\cdots n} = F_w$, so Theorem~\ref{thm:positroid-variety-class} gives
\begin{equation*}
    [\Sigma_{M(w)}] = [\Pi_{\tau^{2n} f_{M(w)}^{-1}}] = \phi(F_{w \times 12\cdots n}) = \phi(F_w).
\end{equation*}
In fact, $\Sigma_{M(w)}$ is a \emph{graph Schubert variety} as defined in \cite[\S 6]{positroidjuggling}, where it is also shown that $[\Sigma_{M(w)}] = \phi(F_w)$.

On the other hand, it is known \cite{plactification} that $s_{D(w)} = F_w$ where $D(w)$ is the \emph{Rothe diagram} of $w$:
\begin{equation*}
D(w) = \{(i,w(j)) \in [n] \times [n] : i < j, w(i) > w(j)\}.
\end{equation*}
For example,
\begin{equation*}
D(3142) = \begin{array}{cccc}
\circ & \circ & \cdot & \cdot\\
\cdot & \cdot & \cdot & \cdot\\
\cdot & \circ & \cdot & \cdot\\
\cdot & \cdot & \cdot & \cdot\\
\end{array}
\end{equation*}
Here we are using $\circ$ for points of $[n] \times [n]$ in $D(w)$ and $\cdot$ for points not in it. We also use matrix coordinates, meaning that $(1,1)$ is at the upper left.

Let $C_w$ be the composition of the (commuting) operators $C_{n+i \to w(i)}$ for $i \in [n]$, acting either on subsets of $[2n]$ or subschemes of $\Gr_n(2n)$ as before.
\begin{thm} For $w \in S_n$, the diagram variety $X_{D(w)}$ is an irreducible component of $C_w \Sigma_{M(w)}$. \end{thm}

\begin{proof} Define
\begin{equation*}
E(w) = ([n] \times [n] \setminus D(w)) \cup \{(i,n+i) : i \in [n]\},
\end{equation*}
so that $X_{D(w)} = \Sigma_{E(w)}$. Since $\codim \Sigma_{M(w)} = \ell(w) = \codim X_{D(w)}$, it suffices by Lemma~\ref{lem:diagram-degeneration} to show that $\Sigma_{C_w M(w)} = \Sigma_{E(w)}$.

Recall that we identify $M(w)$ with the set $\{(i,j) : i \in [n], w(i) \leq j \leq i+n\}$. First, if $j \leq n$ then $(i,j) \notin C_w M(w)$ if and only if $(i,j), (i, w^{-1}(j)+n) \notin M(w)$, if and only if $j < w(i)$ and $i < w^{-1}(j)$, if and only if $(i,j) \in D(w)$: thus $C_w M(w)$ and $E(w)$ agree on $[n] \times [n]$. For instance, $\Sigma_{M(3142)}$ contains
\begin{equation*}
\left\{ \rowspan \begin{bmatrix}
0 & 0 & * & * & * & 0 & 0 & 0\\
* & * & * & * & * & * & 0 & 0\\
0 & 0 & 0 & * & * & * & * & 0\\
0 & * & * & * & * & * & * & *
\end{bmatrix} \right\}
\end{equation*}
as a dense subset, and $C_{3142} \Sigma_{M(3142)}$ accordingly contains
\begin{equation*}
\left\{ \rowspan \begin{bmatrix}
0 & 0 & * & * & * & 0 & 0 & 0\\
* & * & * & * & * & * & 0 & 0\\
* & 0 & * & * & 0 & 0 & * & 0\\
* & * & * & * & * & 0 & * & *
\end{bmatrix} \right\}.
\end{equation*}

As we see in this example, $C_w M(w)$ and $E(w)$ need \emph{not} agree on $[n] \times [n+1,2n]$. However, note that $(i,j+n) \in C_w M(w)$ if and only if $i > j$ and $w(j) > w(i)$, and it is easy to check that this is equivalent to row $j$ of $D(w)$ containing row $i$. Thus, if $A$ is a matrix whose nonzero entries are exactly in positions $C_w M(w)$, then a row operation can be performed on rows $i$ and $j$ which replaces the $*$ in position $(i,j+n)$ by $0$ without changing the pattern of $*$'s in $[n] \times [n]$. This shows that $\Sigma_{C_w M(w)} = \Sigma_{E(w)}$.
\end{proof}

Since $[\lim_{t\to 0} \phi_{t,w} \Sigma_{M(w)}] = [\Sigma_{M(w)}]$, an immediate corollary is an upper bound on $[X_{D(w)}]$.

\begin{thm} \label{thm:diagram-variety-class-bound} $\phi(F_w) - [X_{D(w)}]$ is a nonnegative combination of Schubert classes. \end{thm}

However, this difference of classes can be nonzero. Indeed, the counterexample $D = \{(1,1), (2,2), (3,3), (4,4)\}$ to Conjecture~\ref{conj:liu} discussed in Section~\ref{sec:counterexample} provides an example. Take $w = 21436587$. Then $D(w) = \{(1,1),(3,3),(5,5),(7,7)\}$ can be obtained from $D$ by permuting rows and columns, and viewing $D$ in a larger rectangle. Neither of these operations on diagrams affects $s_D$ or $[X_D]$, identifying the latter with its pullback along the embeddings of $\Gr_k(n)$ into $\Gr_k(n{+}1)$ or $\Gr_{k+1}(n{+}1)$.

\bibliographystyle{plain}
\bibliography{../../bib/algcomb}

\end{document}